\documentclass[10pt]{article}
\usepackage{amsthm}
\usepackage{latexsym}
\usepackage{amsmath}
\usepackage{amssymb}
\usepackage{amsfonts}
\usepackage{graphicx}
\def\div{\operatorname{div}}

\def\curl{\operatorname{curl}}

\def\claim#1{\begin{trivlist}\item[\hskip\labelsep\bf#1]\it}
\def\endclaim{\end{trivlist}}

\newtheorem{theorem}{Theorem}[section]
\newtheorem{lemma}[theorem]{Lemma}

\newtheorem{corollary}[theorem]{Corollary}

\numberwithin{equation}{section}
\theoremstyle{definition}

\begin{document}
\allowdisplaybreaks
\title{ On the inverse elastic scattering by interfaces using one type of scattered waves }
\author{
Manas Kar\thanks{RICAM, Austrian Academy of Sciences,
Altenbergerstrasse 69, A-4040, Linz, Austria.
(Email:manas.kar@oeaw.ac.at)
\newline
Supported by the Austrian Science Fund (FWF): P22341-N18.} \qquad{ Mourad Sini}
\thanks{RICAM, Austrian Academy of Sciences,
Altenbergerstrasse 69, A-4040, Linz, Austria.
(Email:mourad.sini@oeaw.ac.at)
\newline
Partially supported by the Austrian Science Fund (FWF): P22341-N18.}
}
\maketitle

\begin{abstract}
We deal with the problem of the linearized and isotropic elastic inverse scattering by interfaces. We prove that the scattered $P$-parts or
$S$-parts of the far field pattern, corresponding to all the incident plane waves of pressure or shear types, uniquely determine the obstacles for both the penetrable and impenetrable obstacles.
In addition, we state a reconstruction procedure.
In the analysis, we assume only the Lipschitz regularity of the interfaces
and, for the penetrable case, the Lam{\'e} coefficients to be measurable and bounded, inside the obstacles, with the usual jumps across these interfaces.
\end{abstract}

\renewcommand{\theequation}{\thesection.\arabic{equation}}

%\begin{keywords}

%\end{keywords}
%\begin{AMS}
%35P25, 35R30, 78A45.
%\end{AMS}

\begin{section}{\textbf{Introduction}}
Assume $D\subset\mathbb{R}^3$ to be a bounded domain such that $\mathbb{R}^3\setminus\overline{D}$ is connected. 
Let the boundary $\partial D$ to be Lipschitz regular. We assume that the Lam{\'e}
coefficients $\lambda$ and $\mu$ are measurable and bounded and satisfy the conditions $\mu>0$ and $2\mu +3\lambda>0$ and $\mu(x) = \mu_0,$ $\lambda(x) = \lambda_0$ for $x\in \mathbb{R}^3\setminus\overline{D}$ 
with $\mu_0$ and $\lambda_0$ being constants. In addition to that we set $\lambda_D :=\lambda - \lambda_0$ and $\mu_D := \mu-\mu_0$ and assume that $|\mu_D|>0$ and $2\mu_D+3\lambda_D\geq 0$.
We formulate the direct scattering problems as follows.
  Let $u^i$ be an incident field, i.e. a vector field
satisfying $\mu_0 \Delta u^i +(\lambda_0 +\mu_0)\nabla \div u^i+ \kappa^2 u^i=0$ in $\mathbb{R}^3$, where $\kappa$ is the frequency, and $u^s(u^i)$ be the scattered field associated to the incident field $u^i$. 
In the impenetrable case, the scattering problem reads as follows
\begin{equation}  \label{Lame}
\begin{cases}
\mu_0 \Delta u^s +(\lambda_0 +\mu_0)\nabla \div u^s +\kappa^2 u^s =0, \; \mbox{ in }
\mathbb{R}^3\setminus \overline{D}  \\
\sigma(u^s)\cdot \nu=-\sigma(u^i)\cdot \nu,\; \mbox{ on } \partial D  \\
\lim_{\vert x \vert \rightarrow \infty}\vert x \vert (\frac{\partial u_{p}^{s}}{%
\partial {\vert x \vert}}-i\kappa_p u_{p}^{s})=0, \mbox{ and } \lim_{\vert x \vert
\rightarrow \infty}\vert x \vert (\frac{\partial u_{s}^{s}}{\partial {\vert x \vert%
}}-i\kappa_s u_{s}^{s})=0,
\end{cases}
\end{equation}
where the last two limits are uniform in all the directions $\hat{x}:=\frac{x}{\vert x \vert} \in
\mathbb{S}^2$ where $\sigma(u^s)\cdot \nu:=(2\mu \partial_{\nu}+\lambda \nu \div
+\mu \nu \times \curl)u^s$ and the unit normal vector $\nu$ is directed into
the exterior of $D$.
 In the penetrable obstacle case, the total field $u^t :=u^s + u^i$ satisfies
\begin{equation}  \label{Lame_Penetra}
\begin{cases}
\nabla\cdot(\sigma(u^t)) + \kappa^2 u^t = 0, \; \mbox{ in } \mathbb{R}^3  \\
\lim_{\vert x \vert \rightarrow \infty}\vert x \vert (\frac{\partial u_{p}^{s}}{%
\partial {\vert x \vert}}-i\kappa_p u_{p}^{s})=0, \mbox{ and } \lim_{\vert x \vert
\rightarrow \infty}\vert x \vert (\frac{\partial u_{s}^{s}}{\partial {\vert x \vert%
}}-i\kappa_s u_{s}^{s})=0.
\end{cases}
\end{equation}
Let us introduce some further notations. For any displacement field $v$, taken to be a column vector, the corresponding stress tensor $\sigma(v)$ can be represented as a $3\times 3$ matrix:
$
 \sigma(v) = \lambda(\nabla\cdot v)I_3 + 2\mu\epsilon(v),
$
where $I_3$ is the $3\times 3$ identity matrix and $\epsilon(v) = \frac{1}{2}(\nabla v + (\nabla v)^{\top})$ denotes the infinitesimal strain tensor. Note that for $v = (v_1,v_2,v_3)^{\top}$, $\nabla v$ denotes the $3\times 3$ matrix whose
$j$-th row is $\nabla v_j$ for $j= 1,2,3.$ Also for a $3\times 3$ matrix function $A$, $\nabla\cdot A$ denotes the column vector whose $j$-th component is the divergence of the $j$-th row of $A$ for $j=1,2,3.$ \\
Here, we denote $u_{p}^{s}:=-\kappa^{-2}_p \nabla
\div u^s$ to be the longitudinal (or the pressure) part of the field $u^s$ and
$u_{s}^{s}:=\kappa^{-2}_s \mathrm{\curl} \mathrm{\curl} u^s$ to be the
transversal (or the shear) part of the field $u^s$. The constants $\kappa_p:=%
\frac{\kappa}{\sqrt{2\mu_0+\lambda_0}}$ and $\kappa_s:=\frac{\kappa}{\sqrt{\mu_0}}$
are known as the longitudinal and the transversal wave numbers respectively. We
have the well known decomposition of the total field $u$ as the sum of its
longitudinal and transversal parts, i.e. $u=u_p+u_s$. It is well known that
the scattering problems (\ref{Lame}) and \eqref{Lame_Penetra} are well posed using integral equations or variational methods, see for instance 
\cite{Go-Ke, Kup1, Kup2, Ma-Mit} and \cite{ChA}.
The scattered field $u$ has 
the following asymptotic expansion at infinity:
\begin{equation}
u(x):=\frac{e^{i\kappa_p \vert x\vert}}{\vert x \vert} u^\infty_p(\hat{x})+%
\frac{e^{i\kappa_s \vert x\vert}}{\vert x \vert} u^\infty_s(\hat{x})+ O(%
\frac{1}{\vert x \vert^2}), \; \vert x \vert \rightarrow \infty
\end{equation}
uniformly in all the directions $\hat{x}\in \mathbb{S}^2$, see
\cite{A-K} for instance.
The fields $u^\infty_p(\hat{x})$ and $u^\infty_s(\hat{x})$ defined on $%
\mathbb{S}^2$ are called correspondingly the longitudinal and transversal
parts of the far field pattern. The longitudinal part $u^\infty_p(\hat{x})$
is normal to $\mathbb{S}^2$ while the transversal part $u^\infty_s(\hat{x})$
is tangential to $\mathbb{S}^2$. As incident waves, we use pressure (or longitudinal) plane waves or shear (or transversal)
plane waves. They have the analytic forms $u^p_i(x, d):=d e^{i\kappa_p d\cdot x}$
and $u^s_i(x, d):=d^{\perp}e^{i\kappa_s d\cdot x}$ respectively, where $d^{\perp}$
is any vector in $\mathbb{S}^2$ orthogonal to $d$. Remark that $%
u^p_i(\cdot, d)$ is normal to $\mathbb{S}^2$ and $u^s_i(\cdot, d)$ is
tangential to $\mathbb{S}^2$. \\
We denote by $(u_p^{\infty, p}(\cdot, d), u_s^{\infty, p}(\cdot, d))$ the far field pattern associated with the pressure incident field $u^p_i(\cdot, d)$.
Correspondingly, we set $(u_p^{\infty, s}(\cdot, d), u_s^{\infty, s}(\cdot, d))$ to be the far field
pattern associated with the shear incident field $u^s_i(\cdot, d)$.
We write these patterns in a matrix form
\begin{equation}  \label{Farfield-matrix}
 (u^p_i, u^s_i)\mapsto F(u^p_i, u^s_i):=\left[%
\begin{array}{ccc}
u_p^{\infty, p}(\hat x, d) & u_p^{\infty, s}(\hat x, d)   \\
u_s^{\infty, p}(\hat x, d) & u_s^{\infty, s}(\hat x, d)
\end{array}%
\right]
\end{equation}
In this paper, our concern is to show that the knowledge
of any component in \eqref{Farfield-matrix}, for all $(\hat x,d)\in \mathbb{S}^2\times\mathbb{S}^2,$ is enough to determine $D$
and describe a reconstruction procedure. 

From the knowledge of the full farfield map $F$, the first uniqueness result was derived
by Hahner and Hsiao, see \cite{H-S}. Later on the sampling type methods for solving this obstacle inverse scattering problem have been developed by
Alves and Kress \cite{A-K},  Arens \cite{A}, A. Charalambopoulos, D. Gintides
and K. Kiriaki \cite{C-G-K1, C-G-K2, G-K} using the full matrix
(\ref{Farfield-matrix}) for all directions $\hat{x}$ and $d$ in $\mathbb{S}%
^2 $. In \cite{G-S}, D. Gintides and M. Sini, show that any one of the entries in the matrix \eqref{Farfield-matrix}, for all $\hat x, d$ in $\mathbb{S}^2$, is enough. 
In their approach they assumed a $C^4$-regularity of the scatterer to derive the exact asymptotic 
expansion of the so-called probe (or singular sources) indicator function.
Recently in \cite{G-A-M}, Hu, Kirsch and Sini reduced the regularity assumption for the rigid impenetrable obstacles in 3D using the data $u_{p}^{\infty,p}(\cdot,\cdot)$ or $u_{s}^{\infty,s}(\cdot,\cdot)$.

In this paper, we show a systematic way of solving this problem for impenetrable or penetrable obstacles with Lipschitz regularity assumptions on the interfaces using
any component of \eqref{Farfield-matrix}. The analysis is based on some key estimates obtained in \cite{Kar-Sini2} for both the impenetrable obstacle and the penetrable cases.
These estimates combined with some precise analysis of the $P$-parts and $S$-parts of the elastic fundamental tensor allows us 
to justify the needed blow up property of the probe / singular sources indicator functions, see Theorem \ref{main_theo}. 
Regarding the impenetrable case, we consider the free boundary condition. However, as it can be seen, the analysis, based on variational inequalities, can be done for other boundary 
conditions as well (i.e. Dirichlet type, Third type or Forth type boundary conditions, \cite{Kup2}).

The paper is organized as follows.
In Section 2, we state the indicator functions linking the used far field parts to the corresponding $P$-part or $S$-part of the elastic 
fundamental tensor, see \eqref{ppMM}, \eqref{psMM}, \eqref{ssMM} and \eqref{spMM} respectively. 
In Section 3, we state the lower and upper estimates of these indicator functions, see Theorem \ref{main_theo}, and then apply them to show 
the uniqueness results and to justify the reconstruction algorithm (which is based on the probing method, see \cite{I} and \cite{P}). In Section 4 and Section 5,
we justify Theorem \ref{main_theo} for the impenetrable and penetrable obstacle cases respectively. We finish the paper by an Appendix containing
some needed computations concerning the elastic fundamental tensor.

\end{section}
\section{\textbf{The indicator functions linking the used farfield parts to the elastic fundamental tensor}}
  We start with the following identity, see for instance Lemma 3.1 in \cite{A-K}:
\begin{equation}\label{Farfield-Nearfield}
\int_{\partial D} \left( u\cdot \overline{\sigma(v_h)\cdot \nu} -\overline{%
v_h} \cdot \sigma(u) \cdot \nu \right) ds(x)=4\pi\int_{\mathbb{S}%
^2}\left(u^{\infty}_p\cdot \overline{h_p(\hat x)}+u^{\infty}_s\cdot\overline{h_s(\hat x)} \right)ds(\hat x)
\end{equation}
for all radiating fields $u$ with far field pattern $(u_{p}^{\infty}, u_{s}^{\infty})$, where $v_h$ is the Herglotz field with density $h=(h_p, h_s) \in L^2_p(\mathbb{S}%
^2)\times L^2_s(\mathbb{S}^2)$, i.e. $v_h(x):=\int_{\mathbb{S}^2}[e^{i\kappa_p x
\cdot d}h_p(d) +e^{i\kappa_s x \cdot d}h_s(d)]ds(d)$ with $L_{p}^{2}(\mathbb{S}^2) := \{u\in (L^2(\mathbb{S}^2))^3; u(d)\times d=0\}$ while $L_{s}^{2}(\mathbb{S}^2) := \{u\in (L^2(\mathbb{S}^2))^3; u(d)\cdot d=0\}.$
Recall that $\Phi (x,y)$ is the Green's elastic tensor and we set $G_p(x,y)$ and $G_s(x,y)$ as the Green's function associated to the Helmholtz operators i.e. 
$G_t(x,y) = \frac{e^{i\kappa_t|x-y|}}{4\pi|x-y|}$, $t=p$ or $s.$
Also recall that the fundamental tensor of the elasticity is of the form
\begin{equation}\label{funda_elas1}
 \Phi(x,y) := \frac{\kappa_{s}^{2}}{4\pi\kappa^2}\frac{e^{i\kappa_s|x-y|}}{|x-y|}I + \frac{1}{4\pi\kappa^2}\nabla_x\nabla_{x}^{\top}[\frac{e^{i\kappa_s|x-y|}}{|x-y|} - \frac{e^{i\kappa_p|x-y|}}{|x-y|}]
\end{equation}
where $I$ is the identity matrix.
We denote the $p$-part of the elastic fundamental tensor by $\Phi_p(x,y).$ It is of the form
\begin{equation}\label{android}
\begin{split}
\Phi_p(x,y) 
&:= -\frac{1}{\kappa^2}\nabla_x\nabla_{x}^{\top} G_p(x,y) \\
& = -\frac{1}{\kappa^2}\left(\begin{array}{ccc}
                            \frac{\partial^2G_p(x,y)}{\partial x_{1}^{2}} & \frac{\partial^2G_p(x,y)}{\partial x_2\partial x_1} & \frac{\partial^2G_p(x,y)}{\partial x_3\partial x_1} \\
                            \frac{\partial^2G_p(x,y)}{\partial x_1\partial x_2} & \frac{\partial^2G_p(x,y)}{\partial x_{2}^{2}} & \frac{\partial^2G_p(x,y)}{\partial x_3\partial x_2} \\
                            \frac{\partial^2G_p(x,y)}{\partial x_1\partial x_3} & \frac{\partial^2G_p(x,y)}{\partial x_2\partial x_3} & \frac{\partial^2G_p(x,y)}{\partial x_{3}^{2}}
                           \end{array}
                           \right) \\
& =: (\Phi_{p}^{1},\Phi_{p}^{2},\Phi_{p}^{3}),                     
\end{split}
\end{equation}
where $\Phi_{p}^{j}, j=1,2,3 $ are the column vectors of the $p$-part of the elastic fundamental tensor.
The $s$-part of the elastic fundamental tensor denoted by $\Phi_s(x,y)$ is of the form
\begin{equation}
\begin{split}
\Phi_s(x,y) 
&:=\frac{1}{\kappa^2}\curl_x\curl_x(G_s(x,y)I) \\
& = \frac{\kappa_{s}^{2}}{\kappa^2}G_s(x,y)I + \frac{1}{\kappa^2}\nabla_x\nabla_{x}^{\top} G_s(x,y) \\
& = \frac{1}{\kappa^2}\left(\begin{array}{ccc}
                             k_{s}^{2}G_s+\frac{\partial^2G_s}{\partial x_{1}^{2}} & \frac{\partial G_s}{\partial x_2\partial x_1} & \frac{\partial^2G_s}{\partial x_3\partial x_1} \\
                             \frac{\partial^2G_s}{\partial x_1\partial x_2} & k_{s}^{2}G_s+\frac{\partial^2G_s}{\partial x_{2}^{2}} & \frac{\partial^2G_s}{\partial x_3\partial x_2} \\
                             \frac{\partial^2 G_s}{\partial x_1\partial x_3} & \frac{\partial^2G_s}{\partial x_2\partial x_3} & k_{s}^{2}G_s+ \frac{\partial^2G_s}{\partial x_{3}^{2}}
                            \end{array}
                            \right) \\
& =: (\Phi_{s}^{1},\Phi_{s}^{2},\Phi_{s}^{3}),                            
\end{split}
\end{equation}
where $\Phi_{s}^{j},j=1,2,3$ are the column vectors of the $s$-part of the elastic fundamental tensor.
Note that both $\Phi_{p}^{j}$ and $\Phi_{s}^{j}$ satisfy $\mu_0 \Delta u +(\lambda_0 +\mu_0)\nabla \div u+ \kappa^2 u=0$ for $x\neq y$ and $j=1,2,3$.
Let $y\in\mathbb{R}^3\setminus\overline{D}.$ Consider a $C^2$-smooth domain $B$ such that $D\subset\subset B$ and $y\notin B$. 
Now we define the Herglotz wave operator $H : (L^2(\mathbb{S}^2))^3 \rightarrow (L^2(\partial B))^3$ corresponding to the Lam{\'e} model by $(Hh)(x) := v_h(x).$
We can find a sequence of densities $(h_{n,j}^{p})_n$ and $(h_{n,j}^{s})_n$ such that the Herglotz waves $v_{h_{n,j}^{p}}$ and $v_{h_{n,j}^{s}}$ converges to $\Phi_{p}^{j}(\cdot,y)$
and $\Phi_{s}^{j}(\cdot,y)$ respectively on any domain $B$ with $y\notin B, D\subset B\subset\Omega$. These sequences can be obtained as follows.
\subsection{Using longitudinal waves}
We define the Herglotz wave operator $H_p : L^2(\mathbb{S}^2) \rightarrow L^2(\partial B)$ corresponding to the Helmholtz operator $\Delta+\kappa_{p}^{2}$ by $(H_pg)(x) := \int_{\mathbb{S}^2}e^{i\kappa_px\cdot d}g(d)ds(d)$.
We know that if $\kappa_{p}^{2}$ is not a Dirichlet-Laplacian eigenvalue on $\Omega$, then $H_p$ is injective and has a dense range, see \cite{Col}. As the eigenvalues are monotonic
in terms of the domains, then we change $\Omega$ slightly if needed so that $\kappa_{p}^2$ is not an eigenvalue anymore. Note that $y\in \mathbb{R}^3\setminus\overline{B}$.
Hence $G_p(\cdot,y)\in L^2(\partial B)$ and then there exists a sequence $g_{n}^{p}\in L^2(\mathbb{S}^2)$ such that $H_pg_{n}^{p}\rightarrow G_p(\cdot,y)$ in $L^2(\partial B)$ as $n\rightarrow\infty$.
Recall that both $H_pg_n$ and $G_p(\cdot,y)$ satisfy the interior Helmholtz problem in $B$. By the well-posedness of the interior problem and the interior estimate, we deduce
that $H_pg_{n}^{p}\rightarrow G_p(\cdot,y)$ in $C^{\infty}(B)$ since $D\subset\subset B\subset\subset\Omega.$
Hence, see \eqref{android}, $-\frac{1}{\kappa^2}(\frac{\partial^2(H_pg_{n}^{p})}{\partial x_{1}^{2}}, \frac{\partial^2(H_pg_{n}^{p})}{\partial x_1\partial x_2}, \frac{\partial^2(H_pg_{n}^{p})}{\partial x_1\partial x_3})^{\top}
\rightarrow \Phi_{p}^{1}(\cdot,y)$ and then   
$Hh_{n,1}^{p}\rightarrow \Phi_{p}^{1}(\cdot,y)$ in $C^{\infty}(B)$, where 
 $h_{n,1}^{p} := \frac{\kappa_{p}^{2}}{\kappa^2}d_1dg_{n}^{p}(d)$ with $d=(d_1,d_2,d_3)^{\top}$. 
 
Let $(u_{p}^{\infty,p},u_{s}^{\infty,p})$ be the far field associated to the incident field $de^{i\kappa_pd\cdot x}.$
By the principle of superposition, the far field associated to the incident field
 \begin{equation*}
v_{h_{n,1}^{p}}(x):=\int_{\mathbb{S}^2}\frac{\kappa_{p}^{2}}{\kappa^2}d_1g_{n}^{p}(d)de^{i\kappa_p d \cdot x}ds(d)
 \end{equation*}
is given by
\begin{equation*}
\begin{split}
u_{h_{n,1}^{p}}^{\infty}(\hat x)
& :=(u_{h_{n,1}^{p}}^{\infty, p}(\hat x), u_{h_{n,1}^{p}}^{\infty,s}(\hat x)) \\
& =(\int_{\mathbb{S}^2}u_p^{\infty,p}(\hat x, d)\frac{\kappa_{p}^{2}}{\kappa^2}d_1g_{n}^{p}(d)ds(d), \int_{%
\mathbb{S}^2}u_s^{\infty,p}(\hat x, d)\frac{\kappa_{p}^{2}}{\kappa^2}d_1g_{n}^{p}(d)ds(d)).
\end{split}
\end{equation*}
From \eqref{Farfield-Nearfield}, we obtain
\begin{align}
& \frac{4\pi\kappa_{p}^{4}}{\kappa^4}\int_{\mathbb{S}^2}\int_{\mathbb{S}^2}[u_{p}^{\infty,p}(\hat x, d)d_1g_{n}^{p}(d)]\cdot[\overline{\hat x_1\hat x g_{n}^{p}(\hat x)}]ds(d)ds(\hat x) \nonumber \\
& = \int_{\partial D}[u^s(v_{h_{n,1}^{p}}(x))\cdot (\overline{\sigma(v_{h_{n,1}^{p}}(x))}\cdot \nu(x))-\overline{v_{h_{n,1}^{p}}(x)}\cdot(\sigma(u^s(v_{h_{n,1}^{p}}(x)))\cdot \nu(x))]ds(x)
 \end{align}
where $u^s(v_{h_{n,1}^{p}}(x))$ is the scattered field associated to the Herglotz field $v_{h_{n,1}^{p}}$. The dot $\cdot$ in the left hand side is vector product. 
Now, using the fact that $v_{h_{n,1}^{p}}(x)\rightarrow \Phi_{p}^{1}(x,y)$ in $C^{\infty}(B),$ the trace theorem and the well posedness of the scattering
problem, we obtain
\begin{align}
& \frac{4\pi\kappa_{p}^{4}}{\kappa^4}\lim_{n\rightarrow\infty}\int_{\mathbb{S}^2}\int_{\mathbb{S}^2}[u_{p}^{\infty,p}(\hat x, d)d_1g_{n}^{p}(d)]\cdot[\overline{\hat x_1\hat x g_{n}^{p}(\hat x)}]ds(d)ds(\hat x) \nonumber \\
& = \int_{\partial D}[u^s(\Phi_{p}^{1}(x,y))\cdot (\overline{\sigma(\Phi_{p}^{1}(x,y))}\cdot \nu(x))-\overline{\Phi_{p}^{1}(x,y)}\cdot(\sigma(u^s(\Phi_{p}^{1}(x,y)))\cdot \nu(x))]ds(x).
 \end{align}
Similarly, we can find sequences of other Herglotz fields $h_{n,j}^{p}$ so that the sequence converges to $\Phi_{p}^{j}$ with $j=2,3$.
Applying the steps, we obtain
\begin{align}
& \frac{4\pi\kappa_{p}^{4}}{\kappa^4}\lim_{n\rightarrow\infty}\int_{\mathbb{S}^2}\int_{\mathbb{S}^2}[u_{p}^{\infty,p}(\hat x, d)d_jg_{n}^{p}(d)]\cdot[\overline{\hat x_j\hat x g_{n}^{p}(\hat x)}]ds(d)ds(\hat x) \nonumber \\
& = \int_{\partial D}[u^s(\Phi_{p}^{j}(x,y))\cdot (\overline{\sigma(\Phi_{p}^{j}(x,y))}\cdot \nu(x))-\overline{\Phi_{p}^{j}(x,y)}\cdot(\sigma(u^s(\Phi_{p}^{j}(x,y)))\cdot \nu(x))]ds(x)
 \end{align}
 for $j=2,3.$
Hence\footnote{Due to some singularity issues we need to sum up all the corresponding terms, see the proof of Lemma \ref{esti_p}.}
\begin{align}
& \frac{4\pi\kappa_{p}^{4}}{\kappa^4}\lim_{n\rightarrow\infty}\sum_{j=1}^{3}\int_{\mathbb{S}^2}\int_{\mathbb{S}^2}[u_{p}^{\infty,p}(\hat x, d)d_jg_{n}^{p}(d)]\cdot[\overline{\hat x_j\hat x g_{n}^{p}(\hat x)}]ds(d)ds(\hat x) \nonumber \\
& =\sum_{j=1}^{3}\int_{\partial D}[u^s(\Phi_{p}^{j}(x,y))\cdot (\overline{\sigma(\Phi_{p}^{j}(x,y))}\cdot \nu(x))-\overline{\Phi_{p}^{j}(x,y)}\cdot(\sigma(u^s(\Phi_{p}^{j}(x,y)))\cdot \nu(x))]ds(x). \label{ppMM}
 \end{align}
Let us now derive a corresponding formula to \eqref{ppMM} for $u_{s}^{\infty,p}$.
Let $y\in\mathbb{R}^3\setminus\overline{D}$. We define the Herglotz wave operator $H_s : L^2(\mathbb{S}^2) \rightarrow L^2(\partial B)$ 
corresponding to the Helmholtz operator $\Delta + \kappa_{s}^{2}$ by $(H_sg)(x) = \int_{\mathbb{S}^2}e^{i\kappa_s x\cdot d}g(d)ds(d)$.
Using the density argument similar as before, we can find a sequence $g_{n}^{s}\in L^2(\mathbb{S}^2)$
such that $H_sg_{n}^{s} \rightarrow G_s(\cdot,y)$ in $L^2(\partial B)$, where $B$ is a $C^2$-smooth bounded domain containing
$D$ and avoiding $y$ $(y\notin B)$ in which the Dirichlet-Laplacian has no eigenvalues. Therefore 
$\frac{1}{\kappa^2}(k_{s}^{2}H_sg_{n}^{s}+\frac{\partial^2}{\partial x_{1}^{2}}(H_sg_{n}^{s}),\frac{\partial^2}{\partial x_1\partial x_2}(H_sg_{n}^{s}),\frac{\partial^2}{\partial x_1\partial x_3}(H_sg_{n}^{s}))^{\top}\rightarrow \Phi_{s}^{1}(\cdot,y)$
in $C^{\infty}(B)$ and hence $Hh_{n,1}^{s}\rightarrow \Phi_{s}^{1}(\cdot,y)$, where 
 $h_{n,1}^{s} := \frac{k_{s}^{2}}{\kappa^2}(e_1-d_1d)g_{n}^{s}(d)$ with $d=(d_1,d_2,d_3)^{\top}$ and $e_1=(1,0,0)^{\top}$.
From \eqref{Farfield-Nearfield}, we obtain
\begin{align}
 & \frac{4\pi\kappa_{p}^{2}\kappa_{s}^{2}}{\kappa^4}\int_{\mathbb{S}^2}\int_{\mathbb{S}^2}[u_{s}^{\infty,p}(\hat x, d)d_1g_{n}^{p}(d)]\cdot [\overline{(e_1-\hat x_1\hat x)g_{n}^{s}(\hat x)}]ds(d)ds(\hat x) \nonumber \\
 &= \int_{\partial D}[u^s(v_{h_{n,1}^{p}}(x))\cdot (\overline{\sigma(v_{h_{n,1}^{s}}(x))}\cdot \nu(x))-\overline{v_{h_{n,1}^{s}}(x)}\cdot(\sigma(u^s(v_{h_{n,1}^{p}}(x)))\cdot \nu(x))]ds(x)
\end{align}
where $u^s(v_{h_{n,1}^{p}})$ be the scattered field associated to the Herglotz wave $v_{h_{n,1}^{p}}.$ Using the fact that $v_{h_{n,1}^{p}}\rightarrow \Phi_{p}^{1}(\cdot,y)$
and $v_{h_{n,1}^{s}}\rightarrow \Phi_{s}^{1}(\cdot,y)$ in $C^{\infty}(B)$, the trace theorem and the well-posedness of the scattering problem we obtain
\begin{align}
 & \frac{4\pi\kappa_{p}^{2}\kappa_{s}^{2}}{\kappa^4}\lim_{n\rightarrow\infty}\int_{\mathbb{S}^2}\int_{\mathbb{S}^2}[u_{s}^{\infty,p}(\hat x, d)d_1g_{n}^{p}(d)]\cdot [\overline{(e_1-\hat x_1\hat x)g_{n}^{s}(\hat x)}]ds(d)ds(\hat x) \nonumber \\
 &= \int_{\partial D}[u^s(\Phi_{p}^{1}(x,y))\cdot (\overline{\sigma(\Phi_{s}^{1}(x,y))}\cdot \nu(x))-\overline{\Phi_{s}^{1}(x,y)}\cdot(\sigma(u^s(\Phi_{p}^{1}(x,y)))\cdot \nu(x))]ds(x).
\end{align}
Considering the 2nd and 3rd columns of the $p$-part and $s$-part of the elastic Green's tensor, we obtain the following formulas for $j=2,3$ respectively
\begin{align}
 & \frac{4\pi\kappa_{p}^{2}\kappa_{s}^{2}}{\kappa^4}\lim_{n\rightarrow\infty}\sum_{j=1}^{3}\int_{\mathbb{S}^2}\int_{\mathbb{S}^2}[u_{s}^{\infty,p}(\hat x, d)d_jg_{n}^{p}(d)]\cdot [\overline{(e_j-\hat x_j\hat x)g_{n}^{s}(\hat x)}]ds(d)ds(\hat x) \nonumber \\
 &= \sum_{j=1}^{3}\int_{\partial D}[u^s(\Phi_{p}^{j}(x,y))\cdot (\overline{\sigma(\Phi_{s}^{j}(x,y))}\cdot \nu(x))-\overline{\Phi_{s}^{j}(x,y)}\cdot(\sigma(u^s(\Phi_{p}^{j}(x,y)))\cdot \nu(x))]ds(x). \label{psMM}
\end{align}
\subsection{Using shear incident waves}
Let $(u_{p,j}^{\infty,s}, u_{s,j}^{\infty,s})$ be the far field associated to the incident field $(e_j-d_jd)e^{i\kappa_sd\cdot x}, j=1, 2, 3$.
Observe that $e_j-d_jd\in d^{\top}.$
By the principle of superposition, the farfield associated to the incident field
\[
 v_{g_{n,j}}(x):= \int_{\mathbb{S}^2}\frac{\kappa_{s}^{2}}{\kappa^2}(e_j-d_jd)g_n(d)e^{i\kappa_sd\cdot x} ds(d)
\]
is given by
\begin{align}
 u_{g_{n,j}}^{\infty}(\hat x) 
& := (u_{g_{n,j}}^{\infty,p}(\hat x),u_{g_{n,j}}^{\infty,s}(\hat x)) \nonumber \\
& = (\int_{\mathbb{S}^2}u_{p,j}^{\infty,s}(\hat x, d)\frac{\kappa_{s}^{2}}{\kappa^2}g_n(d)ds(d),\int_{\mathbb{S}^2}u_{s,j}^{\infty,s}(\hat x, d)\frac{\kappa_{s}^{2}}{\kappa^2}g_n(d)ds(d)).
\end{align}
From \eqref{Farfield-Nearfield}, we have
\begin{align}
 & \frac{4\pi\kappa_{s}^{4}}{\kappa^4}\int_{\mathbb{S}^2}\int_{\mathbb{S}^2}[u_{s,j}^{\infty,s}(\hat x, d)g_n(d)]\cdot [\overline{(e_j-\hat x_j\hat x)g_n(\hat x)}]ds(d)ds(\hat x) \nonumber \\
 & = \int_{\partial D}[u^s(v_{h_{n,j}^{s}})\cdot(\overline{\sigma(v_{h_{n,j}^{s}})}\cdot\nu(x))-\overline{v_{h_{n,j}^{s}}}\cdot(\sigma(u^s(v_{h_{n,j}^{s}}))\cdot\nu(x))]ds(x)
\end{align}
where $u^s(v_{h_{n,j}^{s}})$ is the scattered associated to Herglotz field $v_{h_{n,j}^{s}}$. The dot $\cdot$ in the left hand side is the vector product.
Now using the fact $v_{h_{n,j}^{s}} \rightarrow \Phi_{s}^{j}(\cdot,y)$ in $C^{\infty}(B),$ the trace theorem and well-posedness of the scattering
problem, we obtain
\begin{align}
 & \frac{4\pi\kappa_{s}^{4}}{\kappa^4}\lim_{n\rightarrow\infty}\int_{\mathbb{S}^2}\int_{\mathbb{S}^2}[u_{s,j}^{\infty,s}(\hat x, d)g_{n}^{s}(d)]\cdot [\overline{(e_j-\hat x_j\hat x)g_{n}^{s}(\hat x)}]ds(d)ds(\hat x) \nonumber \\
 & = \int_{\partial D}[u^s(\Phi_{s}^{j}(x,y))\cdot(\overline{\sigma(\Phi_{s}^{j}(x,y))}\cdot\nu(x))-\overline{\Phi_{s}^{j}(x,y)}\cdot(\sigma(u^s(\Phi_{s}^{j}(x,y)))\cdot\nu(x))]ds(x)
\end{align}
Summing up we obtain
\begin{align}
 & \frac{4\pi\kappa_{s}^{4}}{\kappa^4}\lim_{n\rightarrow\infty}\sum_{j=1}^{3}\int_{\mathbb{S}^2}\int_{\mathbb{S}^2}[u_{s,j}^{\infty,s}(\hat x, d)g_{n}^{s}(d)]\cdot [\overline{(e_j-\hat x_j\hat x)g_{n}^{s}(\hat x)}]ds(d)ds(\hat x) \nonumber \\
 & = \sum_{j=1}^{3}\int_{\partial D}[u^s(\Phi_{s}^{j}(x,y))\cdot(\overline{\sigma(\Phi_{s}^{j}(x,y))}\cdot\nu(x))-\overline{\Phi_{s}^{j}(x,y)}\cdot(\sigma(u^s(\Phi_{s}^{j}(x,y)))\cdot\nu(x))]ds(x) \label{ssMM}
\end{align}
where $\Phi_s = (\Phi_{s}^{1},\Phi_{s}^{2},\Phi_{s}^{3}).$
Similarly we also obtain,
\begin{align}
 & \frac{4\pi\kappa_{s}^{2}\kappa_{p}^{2}}{\kappa^4}\lim_{n\rightarrow\infty}\sum_{j=1}^{3}\int_{\mathbb{S}^2}\int_{\mathbb{S}^2}[u_{p,j}^{\infty,s}(\hat x, d)g_{n}^{s}(d)]\cdot [\overline{\hat x_j\hat x g_{n}^{p}(\hat x)}]ds(d)ds(\hat x) \nonumber \\
 & = \sum_{j=1}^{3}\int_{\partial D}[u^s(\Phi_{s}^{j}(x,y))\cdot(\overline{\sigma(\Phi_{p}^{j}(x,y))}\cdot\nu(x))-\overline{\Phi_{p}^{j}(x,y)}\cdot(\sigma(u^s(\Phi_{s}^{j}(x,y)))\cdot\nu(x))]ds(x). \label{spMM}
\end{align}
\subsection{\textbf{The indicator functions}}\label{indi_sub}
We recall that $(u_{p}^{\infty,p}, u_{s}^{\infty,p})$ correspond to the incident wave $de^{i\kappa_p x\cdot d}$ and $(u_{p,j}^{\infty,s}, u_{s,j}^{\infty,s})$
correspond to the incident wave $(e_j-d_jd)e^{i\kappa_s x\cdot d}, j=1, 2, 3.$ With these data,
we set
\begin{eqnarray}\label{PP}
I_{pp}(y)&:=&\frac{4\pi\kappa_{p}^{4}}{\kappa^4}\lim_{n\rightarrow\infty}\sum_{j=1}^{3}\int_{\mathbb{S}^2}\int_{\mathbb{S}^2}[u_{p}^{\infty,p}(\hat x, d)d_jg_{n}^{p}(d)]\cdot[\overline{\hat x_j\hat x g_{n}^{p}(\hat x)}]ds(d)ds(\hat x),
% \end{equation}
% \begin{equation}
\\ \label{PS}
I_{ps}(y)&:=&\frac{4\pi\kappa_{p}^{2}\kappa_{s}^{2}}{\kappa^4}\lim_{n\rightarrow\infty}\sum_{j=1}^{3}\int_{\mathbb{S}^2}\int_{\mathbb{S}^2}[u_{s}^{\infty,p}(\hat x, d)d_jg_{n}^{p}(d)]\cdot [\overline{(e_j-\hat x_j\hat x)g_{n}^{s}(\hat x)}]ds(d)ds(\hat x),
% \end{equation}
% \begin{equation} 
 \\ \label{SS}
I_{ss}(y)&:=&\frac{4\pi\kappa_{s}^{4}}{\kappa^4}\lim_{n\rightarrow\infty}\sum_{j=1}^{3}\int_{\mathbb{S}^2}\int_{\mathbb{S}^2}[u_{s,j}^{\infty,s}(\hat x, d)g_{n}^{s}(d)]\cdot [\overline{(e_j-\hat x_j\hat x)g_{n}^{s}(\hat x)}]ds(d)ds(\hat x), 
\\ \label{SP} 
\text{and}\hspace{2cm} I_{sp}(y)&:=&\frac{4\pi\kappa_{s}^{2}\kappa_{p}^{2}}{\kappa^4}\lim_{n\rightarrow\infty}\sum_{j=1}^{3}\int_{\mathbb{S}^2}\int_{\mathbb{S}^2}[u_{p,j}^{\infty,s}(\hat x, d)g_{n}^{s}(d)]\cdot [\overline{\hat x_j\hat x g_{n}^{p}(\hat x)}]ds(d)ds(\hat x)\hspace{2cm}
\end{eqnarray}
 where $y\in\mathbb{R}^3\setminus\overline{D}.$
Therefore, the indicator function $I_{pp}$ is defined based on $p$-parts of the far field associated to $p$-incident wave. Correspondingly $I_{ps}$ depends on $s$-part of the far field associated to $p$-incident wave, 
$I_{ss}$ depends on $s$-part of the far field associated to the $s$-incident wave and finally $I_{sp}$ depends on $p$-part of the far field associated to the $s$-incident wave.
\begin{section}{\textbf{Reconstruction scheme and uniqueness}}
The main theoretical result of this work is the following theorem. Using these estimates, i.e \eqref{parochar}, we state a 
reconstruction procedure and in particular we derive a uniqueness result on the identifiablity of the obstacle $D$ from either $p$ or $s$ parts of the far field patterns. 
\begin{theorem}\label{main_theo}
Under the assumption on the problems \eqref{Lame} and \eqref{Lame_Penetra}, described in the introduction, 
we have the following estimate of the indicator functions $I_{ij}(y)$, $(ij) = (ss), (pp), (sp) \ $or$\ (ps)$,
 \begin{equation}\label{parochar}
 c_1 \int_D\frac{1}{|x-y|^8}dx> |I_{ij}(y)|> c_2 \int_D\frac{1}{|x-y|^8}dx + \text{lower order term}, \ y\in \Omega\setminus\bar D, 
 \end{equation}
where $c_1$ and $c_2$ are positive constants independent on $y$.
\end{theorem}
\subsection{Reconstruction procedure}
We describe a procedure to reconstruct $D$ based on Theorem \ref{main_theo}. We proceed in the following steps.
\begin{enumerate}
 \item 
 Let us consider $\Omega$ as large known domain such that $D\subset\Omega$.
\item We start by taking a point $y\in\Omega$ but located near $\partial\Omega$.
\item Take a domain $B\subset\Omega$ such that $y\notin B$ and $\partial B$ close enough to $\partial\Omega$ so that
$D\subset B.$
\item Solve the integral equation of the first kind
\[
 H_pg = G_p(\cdot,y) \ \text{on}\ \partial B.
\]
Note that the above equation is ill-posed. So, we apply regularization methods and select the
solution $g := g_{m}^{y},$ where $m$ is related to the regularization parameter. 
\item From the information of the far field $u_{p}^{\infty,p}(\hat x, d)$
and from Step 4, we calculate the indicator function
\[
 I_{pp}(y):=\frac{4\pi\kappa_{p}^{4}}{\kappa^4}\lim_{n\rightarrow\infty}\sum_{j=1}^{3}\int_{\mathbb{S}^2}\int_{\mathbb{S}^2}[u_{p}^{\infty,p}(\hat x, d)d_jg_{n}^{p}(d)]\cdot[\overline{\hat x_j\hat x g_{n}^{p}(\hat x)}]ds(d)ds(\hat x).
\]
\item If $|I_{pp}(y)|$ is not so large, then $y$ is away from $\partial D.$ In this case take another point near $y$
and apply the step 1, 2, 3, 4.
\item If $|I_{pp}(y)|$ is large then $y$ is near $\partial D.$
In this case we select the point $y$. Therefore we approximate $\partial D$ by collecting all these selected points.
\end{enumerate}
We can also reconstruct the obstacle $D$ using the other types of the far field data.
\subsection{A uniqueness result}
Using Theorem \ref{main_theo}, we obtain the following uniqueness result.
\begin{corollary}
Let $D$ and $\tilde D$ be two scatterers having Lipschitz regular boundaries such that
\[
  u_{p}^{\infty,p}(d,\hat x, D) = u_{p}^{\infty,p}(d,\hat x, \tilde D), \ \text{for all}\ \hat x, d \in \mathbb{S}^2 
\]
 then $D = \tilde D.$ The results holds also for the other types of far field data.
\end{corollary}
\begin{proof}{of the corollary.}
We prove this corollary by the standard Isakov's contradiction argument. Suppose that $D\neq \tilde D$ and $D\cup\tilde D\subset\Omega.$
Hence, we have a point $z\in\partial D$ such that $z\notin\tilde D.$ Let $y\in\Omega\setminus\overline{(D\cup\tilde D)}$ and select
a sequence $(g̣_{n}^{p})_{n\in\mathbb{N}}$ as in step 4 in the reconstruction scheme.  
Since $u_{p}^{\infty,p}(d,\hat x, D) = u_{p}^{\infty,p}(d,\hat x, \tilde D)$, we obtain
\[
  I_{pp}(y) = \tilde I_{pp}(y),
\]
where $I_{pp}(y)$ and $\tilde I_{pp}(y)$ are the indicator functions corresponding to $D$ and $\tilde D$ respectively.
From Theorem \ref{main_theo}, we have for $y\in\Omega\setminus\overline{(D\cup\tilde D)}$,
\begin{equation}\label{uniFinN}
 \begin{split}
 c_1\int_{\tilde D}\frac{1}{|x-y|^8} dx \geq |\tilde{I}_{pp}(y)| 
&= |I_{pp}(y)| \\
 & \geq c_2\int_D\frac{1}{|x-y|^8} dx + \text{lower order terms} \\
 & \geq c_2[d(y,D)]^{-5} + \text{lower order terms}.
 \end{split}
\end{equation}
Based on \eqref{uniFinN}, we observed the following.
When $y$ approaches to $z$ the indicator function $I_{pp}(y)$
$D$ blows up to infinity since $z\in\partial D$. On the other hand, as $z\notin\overline{\tilde D}$, then 
the indicator function $\tilde{I}_{pp}(y)$ is finite.
This contradicts the fact that $I_{pp}(y) = \tilde{I}_{pp}(y)$, $\forall y\in\Omega\setminus\overline{(D\cup\tilde D)}$.
Hence $D = \tilde D.$
\end{proof}
\end{section}
\section{\textbf{Proof of Theorem \ref{main_theo} for the impenetrable case}}
 In this section, we prove Theorem \ref{main_theo} in the impenetrable obstacle case.
We set
\begin{equation}\label{General form}
 I(v, w):=\int_{\partial D}[u^s(v)\cdot (\overline{\sigma(w)}\cdot \nu)-\overline{w}\cdot(\sigma(u^s(v))\cdot \nu)]ds(x)
\end{equation}
 where $u^s(v)$ is the scattered field associated to the incident field $v$ and $v,w$ are assumed to be column vectors which satisfy the Lam{\'e} system in domains containing $\overline{D}$.
 Hence from \eqref{ppMM}, \eqref{psMM}, \eqref{ssMM} and \eqref{spMM}, we have
 \[
  I_{pp}(y) = \sum_{j=1}^{3}I(\Phi_{p}^{j},\Phi_{p}^{j}), \ I_{ps}(y) = \sum_{j=1}^{3}I(\Phi_{p}^{j},\Phi_{s}^{j}), \ I_{ss}(y) = \sum_{j=1}^{3}I(\Phi_{s}^{j},\Phi_{s}^{j}), \ I_{sp}(y) = \sum_{j=1}^{3}I(\Phi_{s}^{j},\Phi_{p}^{j}).
 \]
Since, $y\notin\overline{D}$, then both $\Phi_{p}^{j}$ and $\Phi_{s}^{j}$ satisfy the Lam{\'e} system in $\mathbb{R}^3\setminus\{y\}(\supset\overline{D})$.
Using integration by parts and the boundary conditions, we can write
\[
\begin{split}
 &\int_{\partial D}u^s(v)\cdot (\overline{\sigma(w)}\cdot \nu)ds(x) \\
 & = -\int_{\partial D}u^s(v)\cdot (\overline{\sigma(u^s(w))}\cdot \nu)ds(x) \\
 &=-\int_{\partial \Omega}u^s(v)\cdot (\overline{\sigma(u^s(w))}\cdot \nu)ds(x)+\int_{\Omega \setminus{\overline{D}}} \overline{\sigma(u^s(w))}\cdot (\nabla u^s(v))^\top dx -
\kappa^2 \int_{\Omega \setminus\overline{D}} u^s(v)\cdot \overline{u^s(w)}dx \\
\end{split}
\]
and
$$
\int_{\partial D}\overline{w}\cdot(\sigma(u^s(v))\cdot \nu)ds(x)=-\int_{\partial D}\overline{w}\cdot(\sigma(v)\cdot \nu)ds(x)=
-\int_D \sigma(\overline{w})\cdot(\nabla v)^\top dx +\kappa^2 \int_D v\cdot \overline{w}dx.
$$
 Note that here we define the product of two matrices by $A\cdot B=\sum_{j=1}^{3}a_{ij}b_{ij},$ for any matrices $A=(a_{ij})$ and $B=(b_{ij}).$  
Hence \eqref{General form} becomes:
\begin{equation}\label{Identity}
\begin{split}
 I(v, w) =
&-\int_{\partial \Omega}u^s(v)\cdot (\overline{\sigma(u^s(w))}\cdot \nu)ds(x)+\int_{\Omega \setminus\overline{D}} \overline{\sigma(u^s(w))}\cdot(\nabla u^s(v))^\top dx 
 -\kappa^2 \int_{\Omega \setminus\overline{D}} u^s(v)\cdot \overline{u^s(w)}dx \\
& +\int_D \overline{\sigma(w)}\cdot(\nabla v)^\top dx -\kappa^2 \int_D v\cdot \overline{w}dx.\\
\end{split}
\end{equation}
\textbf{Key inequalities for $I_{ss}$ and $I_{pp}$}:
In this case, we take $v=w$. Hence, we have
$$
I(v,v)\geq -\int_{\partial \Omega}u^s(v)\cdot (\overline{\sigma(u^s(v))}\cdot \nu)ds(x)+\int_D \overline{\sigma(v)}\cdot(\nabla v)^\top dx-
\kappa^2 \int_{\Omega \setminus\overline{D}} \vert u^s(v)\vert^2dx -\kappa^2 \int_D \vert v\vert^2dx.
$$
By the ellipticity condition of the elasticity tensor and the Korn inequality, we obtain
$$
\int_D \overline{\sigma(v)}\cdot (\nabla v)^\top dx =
\int_D \overline{\varrho(\epsilon(v))}\epsilon(v)dx\geq c_1 \int_D \overline{\epsilon(v)}\cdot \epsilon(v)dx\geq \frac{c_1}{C_K}\Vert \nabla v\Vert^2_{L^2(D)}-
c_1\Vert v\Vert^2_{L^2(D)}
$$
where $C_K$ is the Korn constant and $\varrho$ is the elasticity tensor.
Hence
\begin{equation}\label{Omega-v}
I(v,v)\geq -\int_{\partial \Omega}u^s(v)\cdot (\overline{\sigma(u^s(v))}\cdot \nu)ds(x)+c_2\Vert \nabla v\Vert^2_{L^2(D)}-
\kappa^2 \int_{\Omega \setminus\overline{D}} \vert u^s(v)\vert^2dx -(\kappa^2+c_1) \int_D \vert v\vert^2dx.
\end{equation}
\textbf{Key inequalities for $I_{sp}$ and $I_{ps}$}:
In this case, we take $v\neq w$. We use then the form:
$
I(v, w)= -I(v,v) + I(v,U)
$
where $U := v+w.$ 
Using the well posedness of the forward scattering problem and the trace theorem, we show that
$
 |I(v,U)| \leq C\left(\frac{1}{\epsilon}\|\nabla U\|_{L^2(D)}^{2} + \epsilon\|\nabla v\|_{L^2(D)}^{2}\right) \ \text{for} \ 0<\epsilon\ll1.
$
Combining this estimate with \eqref{Omega-v}, we obtain
\begin{equation}\label{imps}
 -I(v,w) \geq I(v,v) - C\left(\frac{1}{\epsilon}\|\nabla U\|_{L^2(D)}^{2} + \epsilon\|\nabla v\|_{L^2(D)}^{2}\right).
\end{equation}
In the following lemma we estimate the boundary integral term and the scattering term in \eqref{Omega-v} by terms involving only the incident wave.
\begin{lemma}\label{lem_im}
Let $v$ be solution of the Lam{\'e} system in a domain containing $\overline{D}$
and let $u^s(v)$ be the corresponding scattered wave, i.e. solution of \eqref{Lame} replacing $u^i$ by $v$.
For $\frac{1}{2}\leq t <1,$ we have the following two estimates 
\begin{align}
& \vert \int_{\partial \Omega}u^s(v)\cdot (\overline{\sigma(u^s(v))}\cdot \nu)ds(x)
\vert  \leq C \Vert v\Vert^2_{H^{-t+\frac{3}{2}}(D)}. \label{someTHH} \\
& \Vert u^s(v)\Vert^2_{L^2(\Omega \setminus\overline{D})}\leq \Vert v\Vert^2_{H^{-t+\frac{3}{2}}(D)}. \label{DiiiDD}
 \end{align}
\end{lemma}
\begin{proof}
 See (\cite{Kar-Sini2}, subsection 5.2.2).
\end{proof}
  Now we choose $\frac{1}{2}<t<1$, then by interpolation and using the Young inequality we obtain:
\begin{equation}\label{us--espilon}
 \Vert v\Vert^2_{H^{-t+\frac{3}{2}}(D)}\leq \epsilon \Vert \nabla v \Vert^2_{L^2(D)}+\frac{C}{\epsilon}\Vert v \Vert^2_{L^2(D)}  
\end{equation}
with some $C>0$ fixed and every $\epsilon>0$.
Therefore combining \eqref{Omega-v}, Lemma \ref{lem_im} and \eqref{us--espilon}, we deduce that
\begin{equation}\label{main-estimate}
 -I(v, v) \geq c\Vert \nabla v \Vert^2_{L^2(D)} -C\Vert v \Vert^2_{L^2(D)}, \; \tau >>1.
\end{equation}
In the case when $v\neq w,$ we use the form
\begin{equation}\label{fo}
 I(v, w)= -I(v,v) + I(v,U)
\end{equation}
 where $U = v+w.$ 
Combining the estimate
\[
 |I(v,U)| \leq C\left(\frac{1}{\epsilon}\|\nabla U\|_{L^2(D)}^{2} + \epsilon\|\nabla v\|_{L^2(D)}^{2}\right) \ \text{for} \ 0<\epsilon\ll1
\]
together with \eqref{main-estimate} and the form \eqref{fo}, we deduce that
\begin{equation}\label{sp_im}
 I(v,w) \geq (c-\epsilon)\|\nabla v\|_{L^2(D)}^{2} - C\|v\|_{L^2(D)}^{2} - \frac{c_1}{\epsilon}\|\nabla U\|_{L^2(D)}^{2}.
\end{equation}
\begin{lemma}\label{esti_p}
 Let $y\in\mathbb{R}\setminus\overline{D}$. We have the following estimates
\begin{enumerate}
 \item 
\begin{equation}\label{amarMUKti}
\|\Phi_{p}^{j}(\cdot,y)\|_{L^2(D)}^{2}\leq C\int_D\frac{1}{|x-y|^6}dx
\end{equation}
\item
\begin{equation}\label{TUMIi}
 \sum_{j=1}^{3}\|\nabla \Phi_{p}^{j}(\cdot,y)\|_{L^2(D)}^{2} \geq C\int_D\frac{1}{|x-y|^8}dx + \text{lower order terms}
\end{equation}
recalling that $\Phi_{p}^{j}$ is the $j$-th column of the $p$-part of the fundamental tensor of the elasticity for all $j=1, 2, 3$.
The estimates \eqref{amarMUKti} and \eqref{TUMIi} are valid for $\Phi_{s}^{j}, j=1 ,2 ,3$, the $j$-th column of the corresponding $s$-parts. 
\item
For all $j=1,2,3$
\[
 \|\nabla\Phi^j(\cdot,y)\|_{L^2(D)}^{2}\leq C\int_D\frac{1}{|x-y|^4}dx 
\]
where $\Phi^j$ is the $j$-th column of the fundamental tensor of the elasticity.
\end{enumerate}
\end{lemma}
\begin{proof}
 \begin{enumerate}
  \item 
From the explicit form of $\Phi_{p}^{j}$ we obtain
\begin{align}
 \|\Phi_{p}^{j}\|_{L^2(D)}^{2}
& = \frac{1}{(\kappa^2)^2}\left[\|\frac{\partial^2G_p}{\partial x_j\partial x_1}\|_{L^2(D)}^{2}+\|\frac{\partial^2G_p}{\partial x_j\partial x_2}\|_{L^2(D)}^{2}+\|\frac{\partial^2G_p}{\partial x_j\partial x_3}\|_{L^2(D)}^{2}\right] \nonumber\\
& \leq C \left[\int_D\frac{1}{|x-y|^2}dx+\int_D\frac{1}{|x-y|^4}dx+\int_D\frac{1}{|x-y|^6}dx\right].
\end{align}
\item
To estimate the gradient term, we need to write clearly the explicit form of the gradient term which is of the form
\[
 \nabla\Phi_{p}^{j}=-\frac{1}{\kappa^2}\left(\begin{array}{ccc}
                           \frac{\partial^3G_p}{\partial x_1\partial x_j\partial x_1} & \frac{\partial^3G_p}{\partial x_2\partial x_j\partial x_1} & \frac{\partial^3G_p}{\partial x_3\partial x_j\partial x_1} \\
                           \frac{\partial^3G_p}{\partial x_1\partial x_j\partial x_2} & \frac{\partial^3G_p}{\partial x_2\partial x_j\partial x_2} & \frac{\partial^3G_p}{\partial x_3\partial x_j\partial x_2} \\
                           \frac{\partial^3G_p}{\partial x_1\partial x_j\partial x_3} & \frac{\partial^3G_p}{\partial x_2\partial x_j\partial x_3} & \frac{\partial^3G_p}{\partial x_3\partial x_j\partial x_3} 
                          \end{array}
\right).
\]
Its norm can be written as
\[
 \|\nabla\Phi_{p}^{j}\|_{L^2(D)}^{2} = \frac{1}{\kappa^4}\sum_{l,m=1}^{3}\int_D|\frac{\partial^3G_p}{\partial x_l\partial x_j\partial x_m}|^2dx.
\]
The idea of the proof is as follows. First we have to look for the dominating term of the each entry of this matrix and second we use the $\epsilon$-inequality.
Let us consider the term $\frac{\partial^3G_p}{\partial x_{l}^{3}}$ for $l=1,2,3.$ The dominating term of $\frac{\partial^3G_p}{\partial x_{l}^{3}}$ is 
$\frac{1}{4\pi}e^{i\kappa_p|x-y|}[-\frac{3(x_l-y_l)}{|x-y|^5}+15\frac{(x_l-y_l)^3}{|x-y|^7}],$ for $l=1,2,3$ see \eqref{ApGreNde1} in the appendix. Now using $\epsilon$-inequality (precisely for any $a,b$ we know that
$(a-b)^2\geq (1-\epsilon)a^2+(1-\frac{1}{\epsilon})b^2$) we have for $l=1,2,3$
\[
 |\frac{\partial^3G_p}{\partial x_{l}^{3}}|^2\geq C \left[\frac{3(x_l-y_l)}{|x-y|^5}-15\frac{(x_l-y_l)^3}{|x-y|^7}\right]^2+ \text{lower order terms}
\]
Similarly for $l\neq m,$ with $l,m = 1,2,3$, compare \eqref{ApGreNde2}, we have
\begin{align*}
 |\frac{\partial^3G_p}{\partial x_l\partial x_{m}^{2}}|^2
&=|\frac{\partial^3G_p}{\partial x_{m}^{2}\partial x_l}|^2
=|\frac{\partial^3G_p}{\partial x_m\partial x_l\partial x_m}|^2 \\
&\geq C \left[\frac{3(x_l-y_l)}{|x-y|^5}-15\frac{(x_m-y_m)^2(x_l-y_l)}{|x-y|^7}\right]^2+ \text{lower order terms},
\end{align*}
and for $k\neq l\neq m$ with $k,l,m=1,2,3$, compare \eqref{ApGreNde3}, we have
\begin{align}\label{pappaNN}
  |\frac{\partial^3G_p}{\partial x_k\partial x_l\partial x_m}|^2\geq C \left[-15\frac{(x_k-y_k)(x_l-y_l)(x_m-y_m)}{|x-y|^7}\right]^2+ \text{lower order terms}.
\end{align}
Therefore,\footnote{We need to sum up all the terms as it can happen that $\|\nabla\Phi_{p}^{j}\|_{L^2(D)}$ has lower estimate then $\int_{D}|x-y|^{-8}dx$, see \eqref{pappaNN}.}
\begin{align*}
 \sum_{j=1}^{3}\sum_{l,m=1}^{3}|\frac{\partial^3G_p}{\partial x_l\partial x_j\partial x_m}|^2 
&\geq C[\frac{9\times 7}{|x-y|^8}-\frac{3\times 90}{|x-y|^{12}}\{\sum_{k=1}^{3}(x_k-y_k)^2\}^2+\frac{(15)^2}{|x-y|^{12}}\{\sum_{k=1}^{3}(x_k-y_k)^2\}^2] \\
& +\text{lower order terms} \\
& = \frac{18C}{|x-y|^8} + \text{lower order terms}.
\end{align*}
Hence,
\begin{align*}
\sum_{j=1}^{3}\|\nabla\Phi_{p}^{j}(\cdot,y)\|_{L^2(D)}^{2}
& = \frac{1}{\kappa^4}\int_D\sum_{j=1}^{3}\sum_{l,m=1}^{3}|\frac{\partial^3G_p}{\partial x_l\partial x_j\partial x_m}|^2dx \\
& \geq C\int_D\frac{1}{|x-y|^8}dx + \text{lower order terms}.
\end{align*}
\item
Note that $\Phi(x,y) = (\Phi_{ij}(x,y))_{i,j}$, where $\Phi_{ij}$ as in \eqref{ij_fund}. For $l\neq i,j$, the term \eqref{kkkij} can be written as
\begin{align*}
 \frac{\partial\Phi_{ij}}{\partial x_l}
= & -\frac{1}{8\pi}\delta_{ij}(\frac{1}{\mu_0}+\frac{1}{\lambda_0+2\mu_0})(x_l-y_l)|x-y|^{-3} \\
& + \frac{1}{4\pi}\sum_{n=1}^{\infty}\frac{i^n}{(n+2)n!}\left(\frac{n+1}{{\mu_0}^{\frac{n+2}{2}}}+ \frac{1}{(\lambda_0+2\mu_0)^{\frac{n+2}{2}}}\right)\kappa^n\delta_{ij}(n-1)(x_l-y_l)|x-y|^{n-3} \\
& - \frac{3}{8\pi}(\frac{1}{\mu_0}-\frac{1}{\lambda_0+2\mu_0})(x_l-y_l)(x_i-y_i)(x_j-y_j)|x-y|^{-5} \\
& -\frac{1}{4\pi}\sum_{n=1}^{\infty}\frac{i^n(n-1)}{(n+2) n!}\left(\frac{1}{{\mu_0}^{\frac{n+2}{2}}} - \frac{1}{(\lambda_0+2\mu_0)^{\frac{n+2}{2}}}\right)\kappa^n(n-3)(x_l-y_l)(x_i-y_i)(x_j-y_j)|x-y|^{n-5}\\
& := A + B + C + D.
\end{align*}
Remark that $A$ and $C$ are the higher order terms and $B$ and $D$ are the convergent series with the lower order terms. Therefore,
the upper bound of $\frac{\partial\Phi_{ij}}{\partial x_l}$ for $l\neq i,j$ can be viewed as 
\begin{align*}
 |\frac{\partial\Phi_{ij}}{\partial x_l}|^2 
 \leq c [A^2 + B^2 + C^2 + D^2] 
 \leq c\frac{1}{|x-y|^4}.
\end{align*}
Similarly, we have
\begin{align*}
& |\frac{\partial\Phi_{ij}}{\partial x_i}|^2 \leq c\frac{1}{|x-y|^4} \\
& |\frac{\partial\Phi_{ij}}{\partial x_j}|^2 \leq c\frac{1}{|x-y|^4}
\end{align*}
where $c$ to be constant.
Summing up we obtain
\begin{align*}
 \|\nabla\Phi^j\|_{L^2(D)}^{2}
& = \int_D\sum_{l=1}^{3}|\nabla\Phi_{lj}(x,y)|^2dx \\
& \leq c \int_D\frac{1}{|x-y|^4}dx.
\end{align*}
 \end{enumerate}
\end{proof}
\subsection{Proof of Theorem \ref{main_theo} for the $I_{pp}$ and $I_{ss}$ cases}
From Lemma \ref{esti_p} and \eqref{main-estimate} we obtain our required estimate.
 \begin{align*}
  -I_{pp}(y)
 =  -\sum_{j=1}^{3}I(\Phi_{p}^{j},\Phi_{p}^{j}) 
& \geq \sum_{j=1}^{3}\left[c\Vert \nabla \Phi_{p}^{j} \Vert^2_{L^2(D)} -C\Vert \Phi_{p}^{j} \Vert^2_{L^2(D)}\right] \\
& \geq C\int_D\frac{1}{|x-y|^8}dx + \text{lower order terms}
 \end{align*}
where $x\neq y, y\in\Omega\setminus\bar D.$
Similarly, we can proof Theorem \ref{main_theo} by using the $s$-part of the fundamental solution.
\subsection{Proof of Theorem \ref{main_theo} for the $I_{sp}$ and $I_{ps}$ cases} 
Combining Lemma \ref{esti_p} and the inequality \eqref{sp_im} we deduce that
 \begin{align*}
   I_{ps}(y)
 =\sum_{j=1}^{3}I(\Phi_{p}^{j},\Phi_{s}^{j}) 
& \geq \sum_{j=1}^{3}\left[(c-\epsilon)\|\nabla \Phi_{p}^{j}\|_{L^2(D)}^{2} - C\|\Phi_{p}^{j}\|_{L^2(D)}^{2} - \frac{c_1}{\epsilon}\|\nabla \Phi^j\|_{L^2(D)}^{2}\right] \\
& \geq c_0\int_D\frac{1}{|x-y|^8}dx + \text{lower order terms}
 \end{align*}
 where $x\neq y, y\in\Omega\setminus\bar D$ and choose $\epsilon>0$ such that $c-\epsilon>c_0>0.$ Similarly, we can prove the estimate for the $sp$ case.
\begin{section}{\textbf{Proof of Theorem \ref{main_theo} for the penetrable case}}
 We consider $v$ as an incident field and $u^s(v)$ the scattered field, therefore the total field $\tilde v = v + u^s(v)$ satisfies the following problem
\begin{equation}  \label{Lame_Penetrable}
\begin{cases}
\nabla\cdot(\sigma(\tilde v)) + \kappa^2 \tilde v = 0, \; \mbox{ in } \mathbb{R}^3  \\
u^s(v) \ \ \text{satisfies the Kupradze radiation condition},
\end{cases}
\end{equation}
recalling that
$
 \sigma(\tilde v) = \lambda(\nabla\cdot\tilde v)I_3 + \mu(\nabla\tilde v + (\nabla\tilde v)^\top).
$
The incident field satisfies 
\begin{equation}\label{cgo_p}
 \nabla\cdot(\sigma_0(v)) + \kappa^2 v = 0 \ \ \text{in} \ \ \Omega,
\end{equation}
where 
$
 \sigma_0(v) = \lambda_0(\nabla\cdot v)I_3 + 2\mu_0\epsilon(v)
$
Accordingly, we will use $\sigma_D(v)$ to denote $\sigma(v)- \sigma_0(v),$ i.e.
$
 \sigma_D(v) = \lambda_D(\nabla\cdot v)I_3 + 2\mu_D\epsilon(v).
$
Note that for a matrix $A=(a_{ij})$, we use $\vert A\vert$ to denote $(\sum_{i,j}|a_{ij}|^{2})^{\frac{1}{2}}$. 
For any matrices $A=(a_{ij})$ and $B=(b_{ij})$, we define the product as 
$ A\cdot B := \sum_{i,j=1}^{3}a_{ij}b_{ij}$. 
As in the impenetrable case, we set
\[
 I(v, w):=\int_{\partial D}[u^s(v)\cdot (\overline{\sigma(w)}\cdot \nu)-\overline{w}\cdot(\sigma(u^s(v))\cdot \nu)]ds(x).
\]
\begin{lemma}\label{funss_1}
 We have the following estimates 
\[
 I(v,w) = - \int_{\Omega}\sigma_D(w)\cdot (\nabla\overline{v})^T dx - \int_{\Omega}\sigma_D(u^s(w))\cdot (\nabla\overline{v})^T dx
\]
\begin{equation}\label{pen_pp}
 \begin{split}
 -I(v,v) \geq
& \int_D \frac{4\mu_0\mu_D}{3\mu} \vert \epsilon(v)\vert^2 dx - \int_D \frac{4\mu_0\mu_D}{9\mu} \vert (\nabla\cdot v)I_3\vert^2 dx - k^2\int_{\Omega} \vert u^s(v)\vert^2dx \\
& - \int_{\partial\Omega} (\sigma(u^s(v))\cdot \nu)\cdot\overline{u^s(v)}ds(x). 
\end{split}
\end{equation}
\end{lemma}
\begin{proof}
 See Lemma 5.2 of \cite{Kar-Sini2}. Note that to prove this lemma we need 
the condition on $\mu_D$ and $\lambda_D$ such that $\mu_D>0$ and $2\mu_D+3\lambda_D\geq 0$, stated in the introduction.
\end{proof}
\begin{lemma}\label{4.3lem}
 Let $v$ be any incident wave and $u^s(v)$ be the scattered wave, i.e. $v+u^s(v)$ is solution of \eqref{Lame_Penetrable}.
 \begin{enumerate}
  \item 
 We have 
\begin{equation}\label{4.3lem0}
 |\int_{\partial\Omega}(\sigma(u^s(v))\cdot\nu)\cdot \overline{u^s(v)}ds(x)| \leq C\mathcal{F}\Vert\nabla v\Vert_{L^2(D)}^{2}
\end{equation}
where $\mathcal{F}$ is defined by 
\[
 \mathcal{F} := \int_{B\setminus\overline{\Omega}}\Vert(\nabla\Phi(x,\cdot))^\top\Vert_{L^2(D)}^{2}dx + \int_{B\setminus\overline{\Omega}}\Vert\nabla(\nabla\Phi(x,\cdot))^\top\Vert_{L^2(D)}^{2}dx,
\]
with $B$ as any smooth domain containing $\overline{\Omega}$.
\item
($L^2-L^q$)-estimate: 
There exists $1\leq q_0 < 2$ such that for $q_0 < q \leq 2,$
\[
 \Vert u^s(v)\Vert_{L^2(\Omega)} \leq C \Vert \nabla v\Vert_{L^q(D)}
\]
with a positive constant $C$.
\end{enumerate}
\end{lemma}
\begin{proof}
 See Lemma 5.3 and Lemma 5.4 of \cite{Kar-Sini2}.
\end{proof}
We first recall Korn's inequality 
\[
 c\|\nabla u\|_{L^2(D)}^{2} \leq \|\epsilon(u)\|_{L^2(D)}^{2} + \|u\|_{L^2(D)}^{2},
\]
for all column vector $u$, where $c>0$ is a constant.
Note that for $j= 1,2,3$
\[
 \nabla\cdot\Phi_{p}^{j} = -\frac{1}{\kappa^2}\frac{\partial}{\partial x_1}(\Delta G_p) = \frac{\kappa_{p}^{2}}{\kappa^2}\frac{\partial G_p}{\partial x_1}.
\]
So, $\nabla\cdot\Phi_{p}^{j}$ behaves as a lower order term. Also, $\nabla\cdot\Phi_{s}^{j} = 0$ for all $j=1,2,3$. 
Hence, applying Korn's inequality, $L^2-L^q$-estimate, Lemma \ref{4.3lem} and the estimate \eqref{pen_pp}
we obtain, for $j=1,2,3$
\begin{align}
 -I(\Phi_{p}^{j},\Phi_{p}^{j}) 
 \geq (c_1-c_5\mathcal{F})\|\nabla\Phi_{p}^{j}\|_{L^2(D)}^{2} - c_2\|\Phi_{p}^{j}\|_{L^2(D)}^{2} - c_3\int_D|\frac{\partial G_p}{\partial x_j}|^2dx
- c_4\|\nabla\Phi_{p}^{j}\|_{L^q(D)}^{2}, \label{ppbhramor}
\end{align}
for $q<2$ and 
\begin{align}
 -I(\Phi_{s}^{j},\Phi_{s}^{j}) 
 \geq (c_1-c_4\mathcal{F})\|\nabla\Phi_{s}^{j}\|_{L^2(D)}^{2} - c_2\|\Phi_{s}^{j}\|_{L^2(D)}^{2}- c_3\|\nabla\Phi_{s}^{j}\|_{L^q(D)}^{2}, \label{ssGopono}
\end{align}
for $q<2$.
For the mixed case, we can write, for $j=1,2,3$
\begin{equation}\label{modhuro}
 \begin{split}
 & I(\Phi_{s}^{j},\Phi_{p}^{j}) = - I(\Phi_{s}^{j},\Phi_{s}^{j}) + I(\Phi_{s}^{j},\Phi^j), \\ 
 & I(\Phi_{p}^{j},\Phi_{s}^{j}) = - I(\Phi_{s}^{j},\Phi_{s}^{j}) + I(\Phi^j,\Phi_{s}^{j}),
\end{split}
\end{equation}
where the $j$-th column of the elastic tensor $\Phi^j = \Phi_{p}^{j}+\Phi_{s}^{j}.$ Using the $\epsilon$-inequality, we have
\begin{equation}\label{janatam}
 |I(\Phi_{s}^{j},\Phi^j)| \leq C\left(\frac{1}{\epsilon}\|\nabla \Phi^j\|_{L^2(D)}^{2} + \epsilon\|\nabla \Phi_{s}^{j}\|_{L^2(D)}^{2} \right).
\end{equation}
Combining \eqref{modhuro} and \eqref{janatam}, we obtain for $q_0<q<2,$
\begin{equation}\label{sp_linear}
 \begin{split}
&  I(\Phi_{s}^{j},\Phi_{p}^{j}), I(\Phi_{p}^{j},\Phi_{s}^{j}) \\
& \geq (C - \mathcal{F} - {\tilde C}{\epsilon})\|\nabla\Phi_{s}^{j}\|_{L^2(D)}^{2} + (C - \mathcal{F} - c_2)\|\Phi_{s}^{j}\|_{L^2(D)}^{2} - c_1\Vert\nabla \Phi_{s}^{j}\Vert_{L^q(D)}^{2} - \frac{\tilde{C}}{\epsilon} \|\nabla \Phi^j\|_{L^2(D)}^{2}.
 \end{split}
\end{equation}
\subsection{Proof of Theorem \ref{main_theo} for the $I_{pp}$ and $I_{ss}$ cases}
Recall that $\Phi^j, \Phi_{p}^{j}$ and $\Phi_{s}^{j}$ are the $j$-th column of the fundamental solution, its $p$-part and $s$-part respectively.
Applying the Minkowski inequality we obtain
\begin{equation}\label{ss_impe}
\begin{split}
 \|\nabla \Phi_{p}^{j}\|_{L^p(D)}^{2}
&\leq C\left[\int_D\frac{1}{|x-y|^{4p}}dx\right]^{\frac{2}{p}} + \text{lower order terms} \\
& \approx \left[d(y,D)\right]^{(3-4p)\frac{2}{p}}
\end{split}
\end{equation}
i,e the term $\|\nabla \Phi_{p}^{j}\|_{L^p(D)}$ has a lower order behavior than the term $\|\nabla \Phi_{p}^{j}\|_{L^2(D)}$ as $p<2.$
Hence from \eqref{ppbhramor}, we have
\[
 \begin{split}
   - I_{pp}(y)
 = -\sum_{j=1}^{3}I(\Phi_{p}^{j},\Phi_{p}^{j}) 
 \geq C\int_D\frac{1}{|x-y|^8}dx + \text{lower order terms}.
 \end{split}
\]
Similarly, using the $s$-part of the fundamental solution of the elasticity we can prove Theorem \ref{main_theo} for $I_{ss}$.
\subsection{Proof of Theorem \ref{main_theo} for the $I_{sp}$ and $I_{ps}$ cases}
From \eqref{sp_linear}, we obtain
 \begin{align*}
    I_{ps}(y)
 = \sum_{j=1}^{3}I(\Phi_{p}^{j},\Phi_{s}^{j}) 
 \geq (C - \mathcal{F} - {\tilde C}{\epsilon})\int_D\frac{1}{|x-y|^8}dx + \text{lower order terms}.
 \end{align*}
Now, Theorem \ref{main_theo} follows by appropriately choosing $\epsilon>0$ and the smooth domain $B$ such that $|B\setminus\overline{\Omega}|$ is small enough so that $(C - \mathcal{F} - {\tilde C}{\epsilon})>c_0>0$. 
Similarly, we can derive the estimate of $I_{sp}(y)$.
\end{section}
\section{Appendix}
\subsection{Derivatives of the Helmholtz fundamental solution}
We have, for $x,y\in \mathbb{R}^3$ with $x\neq y$ 
\[
 G_p(x,y) = \frac{e^{i\kappa_p|x-y|}}{4\pi|x-y|}.
\]
\textbf{1st order partial derivatives}\\
The first partial derivatives of $G_p$ can be written as:
\[
 \frac{\partial G_p(x,y)}{\partial x_l} = \frac{1}{4\pi}e^{i\kappa_p|x-y|}\left[\frac{i\kappa_p(x_l-y_l)}{|x-y|^2}-\frac{(x_l-y_l)}{|x-y|^3}\right],
\]
for all $l=1,2,3.$ \\
\textbf{2nd order partial derivatives}\\
For all $l=1,2,3$
\[
  \frac{\partial^2G_p}{\partial x_{l}^{2}}
= \frac{1}{4\pi}e^{i\kappa_p|x-y|}\left[(i\kappa_p)^2\frac{(x_l-y_l)^2}{|x-y|^3}-3i\kappa_p\frac{(x_l-y_l)^2}{|x-y|^4}+\frac{i\kappa_p}{|x-y|^2}-\frac{1}{|x-y|^3}+3\frac{(x_l-y_l)^2}{|x-y|^5}\right]. 
\]
For $l\neq m$ with $l,m=1,2,3$, we have
\[
 \frac{\partial^2 G_p}{\partial x_m\partial x_l}
= \frac{1}{4\pi}e^{i\kappa_p|x-y|}(x_l-y_l)(x_m-y_m)\left[(i\kappa_p)^2\frac{1}{|x-y|^3}-3i\kappa_p\frac{1}{|x-y|^4}+3\frac{1}{|x-y|^5}\right].
\]
\textbf{3rd order partial derivatives}\\
For all $l=1,2,3$
\begin{align}
 \frac{\partial^3G_p}{\partial x_{l}^{3}} 
=& \frac{1}{4\pi}e^{i\kappa_p|x-y|}(x_l-y_l)[3(i\kappa_p)^2\frac{1}{|x-y|^3}-9(i\kappa_p)\frac{1}{|x-y|^4}+(i\kappa_p)^3\frac{(x_l-y_l)^2}{|x-y|^4}+3\frac{1}{|x-y|^5} \nonumber \\
& -6(i\kappa_p)^2\frac{(x_l-y_l)^2}{|x-y|^5}+15i\kappa_p\frac{(x_l-y_l)^2}{|x-y|^6}-15\frac{(x_l-y_l)^2}{|x-y|^7}]. \label{ApGreNde1}
\end{align}
For $l\neq m$ with $l,m=1,2,3$ we have
\begin{align}
 \frac{\partial^3G_p}{\partial x_l\partial x_{m}^{2}}
=&\frac{\partial^3G_p}{\partial x_m\partial x_l\partial x_m}=\frac{\partial^3G_p}{\partial x_m\partial x_m\partial x_l} \nonumber \\
=& \frac{1}{4\pi}e^{i\kappa_p|x-y|}(x_l-y_l)[(i\kappa_p)^2\frac{1}{|x-y|^3}-3i\kappa_p\frac{1}{|x-y|^4}+(i\kappa_p)^3\frac{(x_m-y_m)^2}{|x-y|^4}+3\frac{1}{|x-y|^5} \nonumber \\
&-6(i\kappa_p)^2\frac{(x_m-y_m)^2}{|x-y|^5}
+15i\kappa_p\frac{(x_m-y_m)^2}{|x-y|^6}-15\frac{(x_m-y_m)^2}{|x-y|^7}]. \label{ApGreNde2}
\end{align}
At last for $k\neq l\neq m$ with $k,l,m=1,2,3$ we have
\begin{align}
 \frac{\partial^3G_p}{\partial x_k\partial x_l\partial x_m}
=& \frac{1}{4\pi}e^{i\kappa_p|x-y|}(x_k-y_k)(x_l-y_l)(x_m-y_m) \nonumber \\
&\left[(i\kappa_p)^3\frac{1}{|x-y|^4}-6(i\kappa_p)^2\frac{1}{|x-y|^5}-9i\kappa_p\frac{1}{|x-y|^6}-15\frac{1}{|x-y|^7}\right]. \label{ApGreNde3}
\end{align}
\subsection{Derivatives of the elastic fundamental tensor}
In \eqref{funda_elas1}, the fundamental tensor of the elastic model is given. Now the $ij$-th element of the fundamental tensor $\Phi(x,y)$ can be viewed as:
\begin{equation}\label{ij_fund}
\begin{split}
 \Phi_{ij}(x,y) 
= & \frac{1}{4\pi}\sum_{n=0}^{\infty}\frac{i^n}{(n+2) n!}\left(\frac{n+1}{{\mu_0}^{\frac{n+2}{2}}}+ \frac{1}{(\lambda_0+2\mu_0)^{\frac{n+2}{2}}}\right)\kappa^n\delta_{ij}|x-y|^{n-1} \\
& - \frac{1}{4\pi}\sum_{n=0}^{\infty}\frac{i^n(n-1)}{(n+2) n!}\left(\frac{1}{{\mu_0}^{\frac{n+2}{2}}} - \frac{1}{(\lambda_0+2\mu_0)^{\frac{n+2}{2}}}\right)\kappa^n|x-y|^{n-3}(x_i-y_i)(x_j-y_j), 
\end{split}
\end{equation}
where $x,y \in \mathbb{R}^3$ with $x\neq y,$ see for more details [\cite{Kup2}, Chap. 2].\\
For $l\neq i,j$
\begin{align}
 \frac{\partial\Phi_{ij}}{\partial x_l}
= & \frac{1}{4\pi}\sum_{n=0}^{\infty}\frac{i^n}{(n+2)n!}\left(\frac{n+1}{{\mu_0}^{\frac{n+2}{2}}}+ \frac{1}{(\lambda_0+2\mu_0)^{\frac{n+2}{2}}}\right)\kappa^n\delta_{ij}(n-1)(x_l-y_l)|x-y|^{n-3} \nonumber \\
& -\frac{1}{4\pi}\sum_{n=0}^{\infty}\frac{i^n(n-1)}{(n+2) n!}\left(\frac{1}{{\mu_0}^{\frac{n+2}{2}}} - \frac{1}{(\lambda_0+2\mu_0)^{\frac{n+2}{2}}}\right)\kappa^n(n-3)(x_l-y_l)(x_i-y_i)(x_j-y_j)|x-y|^{n-5}. \label{kkkij}
\end{align}
Similarly, for $l=i$
\begin{align}
 \frac{\partial\Phi_{ij}}{\partial x_i}
= & \frac{1}{4\pi}\sum_{n=0}^{\infty}\frac{i^n}{(n+2)n!}\left(\frac{n+1}{{\mu_0}^{\frac{n+2}{2}}}+ \frac{1}{(\lambda_0+2\mu_0)^{\frac{n+2}{2}}}\right)\kappa^n\delta_{ij}(n-1)(x_l-y_l)|x-y|^{n-3} \nonumber \\
& -\frac{1}{4\pi}\sum_{n=0}^{\infty}\frac{i^n(n-1)}{(n+2) n!}\left(\frac{1}{{\mu_0}^{\frac{n+2}{2}}} - \frac{1}{(\lambda_0+2\mu_0)^{\frac{n+2}{2}}}\right)\kappa^n[(n-3)(x_l-y_l)(x_i-y_i)(x_j-y_j)|x-y|^{n-5} \nonumber \\
&+(x_j-y_j)|x-y|^{n-3}]. \label{l0i}
\end{align}
and for $l=j$ we obtain
\begin{align}
 \frac{\partial\Phi_{ij}}{\partial x_j}
 =&\frac{1}{4\pi}\sum_{n=0}^{\infty}\frac{i^n}{(n+2)n!}\left(\frac{n+1}{{\mu_0}^{\frac{n+2}{2}}}+ \frac{1}{(\lambda_0+2\mu_0)^{\frac{n+2}{2}}}\right)\kappa^n\delta_{ij}(n-1)(x_l-y_l)|x-y|^{n-3} \nonumber \\
& -\frac{1}{4\pi}\sum_{n=0}^{\infty}\frac{i^n(n-1)}{(n+2) n!}\left(\frac{1}{{\mu_0}^{\frac{n+2}{2}}} - \frac{1}{(\lambda_0+2\mu_0)^{\frac{n+2}{2}}}\right)\kappa^n[(n-3)(x_l-y_l)(x_i-y_i)(x_j-y_j)|x-y|^{n-5} \nonumber \\
&+(x_i-y_i)|x-y|^{n-3}]. \label{l0j} 
\end{align}
% % \bibliographystyle{abbrv}
% % \bibliography{uniqueness}

\end{document}